\documentclass[10pt,reqno]{amsart}
\usepackage{amssymb,amsmath,bm,geometry,graphics,url,color}
\usepackage{amsfonts}
\usepackage{graphicx}
\usepackage{epsfig,multirow}
\usepackage{amscd}
\usepackage{enumerate}
\usepackage{mathrsfs}
\usepackage{xypic}
\usepackage{stmaryrd}
\usepackage{fancybox}
\usepackage{exscale,array}
\usepackage{enumitem}%
\usepackage{color,cases}
\usepackage{hyperref}
\usepackage{algorithm}
\usepackage{algorithmic}
\usepackage{verbatim}
\usepackage{amsmath}
\usepackage{supertabular}
\usepackage{array}
\usepackage{multirow}
\usepackage{booktabs}
\usepackage{threeparttable}

\def\pdt2{\partial_t^2}
\def\pdx2{\partial_x^2}

\newcommand{\bB}{{\bf B}}
\newcommand{\bA}{{\bf A}}

\newcommand{\bS}{\boldsymbol{\sigma}}
\newcommand{\brho}{\boldsymbol{\rho}}

\def\eps{\varepsilon}

\renewcommand{\(}{\left(}
\renewcommand{\)}{\right)}

\newcommand{\bv}{{\bf v}}

\newcommand{\bw}{{\bf w}}

\newcommand{\HH}{{\bf H}}
\newcommand{\EE}{{\bf E}}

\newcommand{\CU}{{\bf curl}}

\newcommand{\BB}{{\bf B}}
\newcommand{\ZZ}{{\bf Z}}

\newcommand{\bx}{\mathbf{x}}

\newcommand{\br}{{\mathbf r}}

\newcommand{\bJ}{\mathbf{J}}

\newtheorem{theo}{Theorem}[section]
\newtheorem{lem}[theo]{Lemma}
\newtheorem{cor}[theo]{Corollary}
\newtheorem{rem}[theo]{Remark}

\newtheorem{prop}[theo]{Proposition}
\newtheorem{ex}[theo]{Example}
\numberwithin{equation}{section}

\title[Closed form solution of Maxwell's equations]{On analytic solution of the Maxwell's equation with non-zero currents}

\author[X.R. Zou]{Xiaorong Zou}\address{\hspace*{-12pt}X.R.~Zou: Global Market Risk Analytic, Bank of America, NYC, NY, USA}
\email{xiaorong.zou@bofa.com}

\begin{document}
\maketitle
\dedicatory{}
\begin{abstract}
An analytic solution has been recently developed for the Maxwell's equation in a medium with zero currents such as vacuum.  The solution is attractive in the sense that it is formulated based on the Fourier expansion of the initial value.  It has been used to study the properties of solutions like certain conservative laws and construct electromagnetic waves with certain features. 

In this paper,  we study Maxwell's equation in a medium with non-zero currents.  The structure of solutions in this setting turns out to be much more complicated than what has been achieved without currents, and a clean structure of analytic solutions as with zero current is no longer available in general. Nevertheless, we can still develop an algorithm to construct the solution effectively.  Our efforts in seeking analytic solution focus on two special cases.  First, we develop analytic solution under the assumption that Ohm's law is satisfied, i.e. the current density is proportional to electronic density;  secondly, we add skew symmetric components under generalized Ohm's law, which is also refereed as Hall effect in literature,  and study the properties of solutions. In addition, we consider the case where an independent local electromagnetic field is included and derive the analytical solution accordingly. As an application, we provide an example to use the analytic solution to construct parallel electronic and magnetic waves. 
\end{abstract}
 \section{Introduction}
The Maxwell's equations form the foundation of classical electromagnetism and play a critical role in a wide variety of applications in science and engineering.   In this paper,  we aim to develop an attractive analytic solution of the Maxwell's equations
 \begin{subequations}\label{maxsys}
\begin{numcases}%
\,\frac{\partial \HH}{\partial t} =-\frac{1}{\mu}\CU\  \EE,\ \ \ \frac{\partial \EE}{\partial t} =\frac{1}{\eps}\CU\  \HH - \bJ,\label{maxsys a}\\
\HH(\bx,0) = H_0(\bx) ,\qquad  \EE(\bx,0) =E_0(\bx), \label{maxsys b}
\end{numcases}
\end{subequations}
where $\bx = (x_1,x_2,x_3)$, $\textbf{H}= (H_1 ,H_2 ,H_3 )^{\intercal}: \mathbb{R}^3\times
\mathbb{R}_+ \rightarrow \mathbb{R}^3$ {stands for}
magnetic field intensity,
$\textbf{E}= (E_1 ,E_2 ,E_3 )^{\intercal}: \mathbb{R}^3\times
\mathbb{R}_+ \rightarrow \mathbb{R}^3$ represents { electric field intensity}, $H_0(\bx), E_0(\bx)$ represents the initial condition of the system, $\bJ$ denote the current density,  $\mu$ and $\eps$ are constant in this paper and characterize the magnetic permeability and electric permittivity, respectively.  It is well known that Maxwell's equations \eqref{maxsys} have unique solutions for all time if the initial values \eqref{maxsys b} are suitably smooth \cite{WW21},\cite{Leis86}. 

There are rich researches on numerical solutions of Maxwell's equations over the last couple of decades due to the importance and diversity of applications \cite{3}-\cite{WJ}. An attractive analytic solution has recently developed in \cite{zou_maxwell} when $\bJ=0$.  Briefly, the Maxwell equation is transformed into the following Abstract Cauchy Problem (ACP) \cite{engel_nagel} of the complex vector field $Z=\EE+ i \HH$ 
\begin{equation}\label{eq:ACP_simple}
\frac{\partial Z(\bx, t)}{\partial t}  = i \cdot \CU  Z(\bx, t),  \quad Z(\bx, 0) = Z_0(\bx): = E_0(\bx) + i H_0(\bx)
\end{equation}
Express $Z_0$ by its Fourier series
\begin{equation}\label{def:z_0}
Z_0(\bx) = \sum_{\bw} a_{\bw} e^{i\bw \cdot \bx} =: \sum_{\bw} Z_{0,\bw}(\bx),
\end{equation}
where $\bw\in R^3$ runs through index set decided by periodicity of $Z_0(\bx)$ and is named a wave vector throughout the paper. Based on the principle of superposition the solution can be decomposed into the summation 
\[
Z(\bx, t) = \sum_{\bw} Z_{\bw}(\bx, t)
\]
where $Z_{\bw}(\bx, t)$ is the solution with $Z_{\bw}(\bx, 0)=Z_{0,\bw}(\bx)$ for a given wave vector $\bw$. It turns out that $Z_{\bw}(\bx, t)$ can be represented as a linear combination of certain base solutions that are associated to eigenvectors of the underlying linear operator $i\cdot \CU$ of ACP.  Hence the solution can be constructed in term of coefficients $\{a_{\bw}\}$ of Fourier series (\ref{def:z_0}) of $Z_0(\bx)$.

In this paper,  we aim to find an analytic solution of Maxwell Equation with non-zero current density $\bJ$ that has following format 
\begin{equation} \label{def:Sigma}
	\bJ  = \bS \EE + J^g(\bx)
\end{equation}
where $ J^g(\bx)$ is an independent local electromagnetic field \cite{bossavit} (page 11) and $\bS$ is the conductivity matrix \cite{buschow}-\cite{Rysti}
\begin{equation}\label{def:bs}
\bS= \left(
\begin{array}{ccc}
	\sigma_{11} & \sigma_{12} & \sigma_{13}\\
	\sigma_{21} & \sigma_{22} & \sigma_{23} \\
	\sigma_{31} & \sigma_{32} & \sigma_{33} \\
\end{array}
\right).
\end{equation}
Note that Eq. (\ref{def:Sigma}) is reduced to the generalized Ohm's Law when $J^g(\bx)$ is dropped out.  

One can follow the same steps as used in the case $\bJ=0$:  transform the equation \ref{maxsys} to an ACP with a underlying linear operator $\bB$, find the base solutions by restricting ACP to a subspace associated to a given wave vector $\bw$, and finally construct the desired solution by the linear combinations of base solutions. 
However, the operator $\bB$ becomes much more complicated since $\EE$ and $\HH$ is not as symmetric as in the case $\bJ=0$.  For example, $\bB$ need to act on a $6$ dim complex vector field whose real part is $(\EE,\HH)$, compared to $3$ dim complex vector space used in the treatment of $\bJ=0$.

We start with the case $J^g(\bx)\equiv 0$ and develop a framework to construct solutions in a general setting of $\bJ$ by Eq. (\ref{def:bs}).  The main result is summarized in Theorem \ref{thm-main} that shows a similar structure of a solution as in the case $\bJ=0$: it is a linear combination of the base solutions with $Z^k_{\bw}e^{i\lambda_{k,\bw}t}$ where $\{\lambda_{k,\bw}\}_{k=1}^6$ is the set of eigenvalues of $\bB_{\bw}$,  the linear operator by restricting $\bB$ to the base space associated to $\bw$, and $\{ Z^k_{\bw}\}_{k=1}^6$ is a set of associated eigenvectors that generates the subspace.  The general result is refined by taking the following conductivity matrix
\begin{equation}\label{def:bs-skew}
	\bS= \left(
	\begin{array}{ccc}
		\sigma & c & -b\\
		-c & \sigma & a \\
		b & -a & \sigma \\
	\end{array}
	\right).
\end{equation}
where $\sigma\ge0$ and $a,b,c$ are three real numbers.  We derive the explicit solutions under Ohm's assumption with $a=b=c=0$ as shown in Theorem \ref{thm-2}. For a general skew parameter $p=(a,b,c)$, we focus on the properties of eigenvalues of $\bB$, which play the key rule how electromagnetic waves propagate.  For example, we prove that eigenvalues are all pure imaginary, associating to the case without energy loss, if and only if $\sigma=0$ (Proposition \ref{prop:real_solution}).  We discuss two special cases where $p$ is either parallel or perpendicular to the wave vector $\bw$ and close-form of eigenvectors are available.

Finally, we consider the general setting \ref{def:Sigma} of $J^g$.  It turns out that the desired analytic solution can be derived from the solution without $J^g(\bx)$ by some adjustment based on the initial conditions.  The result is shown in Theorem \ref{theorem-main-3}.

The rest of the paper is organized as follows. Section \ref{sec:j=general} is used to develop the framework. We start with reviewing certain results about eigenvalues and eigenvectors of $\CU$ operator developed in \cite{zou_maxwell}, which plays a critical tool for the analysis and results developed in this paper.  We then convert Maxwell Equation \ref{maxsys} into an equivalent CAP \ref{eq:ACP_simple} and further study the property of $\bB_{\bw}$, and derive a solution under a general setting as a framework.  Section \ref{sec:j=sigmmaE} is used to develop the desired analytic solution under Ohm's assumption, i.e. $\bJ = \sigma I_3$. It consists three subsections:  Subsection \ref{subsection-general} covers the solution under the subspace associated to a wave vector $\bw$ such that $\sigma\neq 2|\bw|$. Subsection \ref{subsection-special} is devoted to cope with $\sigma= 2|\bw|$ and desired solution is constructed by applying limit of the solution with $\sigma\neq 2|\bw|$ as $\sigma\to 2|\bw|$. The solution with general initial condition is then stated in Subsection \ref{subsection-all_together}. In Section \ref{sec:skew}, we consider the solution under the assumption \ref{def:bs-skew} on $\bJ$. We derive a polynomial of degree $5$ that determines non-zero eigenvalues of $\bB_{\bw}$, and study the properties of the eigenvalues.  We then focus on two special cases where close-form solution of eigenvalues are available.  Section \ref{sec:jg} is used to derive the analytic solution with the general setting for $\bJ$ defined by \ref{def:Sigma}.  Appendix is used to prove several Lemmas and propositions stated in main body.
\section{The framework of analytic solutions under a general setting}\label{sec:j=general}
Follow the notations in \cite{zou_maxwell}, let $\bw =(w_1,w_2, w_3)^T\in R^3$, named as a wave vector in this paper, $V^{\bw}=e^{i\bx \cdot \bw} C^3$ be the $3$-dim complex vector space generated by $e^{i\bx \cdot \bw}$. For any $\eta =\xi e^{i\bx \cdot \bw} \in V^{\bw}$, by definition of $\CU$, we have
\[
\CU (\eta) = i  \Phi_{\bw}\xi e^{i\bx \cdot \bw},
\]
where
\begin{equation}\label{eq:phi}
	\Phi_{\bw}=\left(
	\begin{array}{ccc}
		0 & -w_3 &w_2\\
		w_3 & 0 & -w_1 \\
		-w_2 & w_1 & 0
	\end{array}
	\right),
\end{equation}
Note that $\Phi_{\bw}$ is skew symmetric and has the eigenvalues  $\{0, i|\bw|, -i|\bw|\}$ 
with
\[
|\bw| = \sqrt{w_1^2 + w_2^2 +w_3^2}.
\]
For any vector ${\bw}=(w_1,w_2,w_3)^T\in R^3$, define ${\br}_{\bw} = (w_2-w_3, w_3-w_1, w_1-w_2)^T$ and
\[
s_{\bw} := w_1+w_2+w_3, \quad  \gamma_{\bw} := |{\br}_{\bw}|  = \sqrt{3|\bw|^2 - s_{\bw}^2}, \quad  \nu_{\bw} = \sqrt 2\frac{\gamma_{\bw} }{|\bw|}.
\]
Lemma \ref{prop:curl-1}-\ref{prop:curl-2} summarize some properties of $\CU$ established in \cite{zou_maxwell}.
\begin{lem}\label{prop:curl-1}
	\begin{enumerate}
		\item 
	\begin{equation}\label{eq:Phi_w}
		\CU (\xi e^{i\bx \cdot \bw})= i\lambda \xi e^{i\bx \cdot \bw}  \quad \mbox{if and only if} \quad	\Phi_{\bw}\xi = \lambda \xi.
	\end{equation}
	\item 
 The restriction of $\CU$ on $V^{\bw}$ has three eigenvalues  $\{\mu_{d,\bw}\}_{d=1}^3 $ with
\begin{equation}\label{eq:eigenvalue}
\mu_{1,\bw} = 0, \quad \mu_{2,\bw} = -|\bw|, \quad \mu_{3,\bw} = -\mu_{2,\bw},
\end{equation}
	\end{enumerate}
\end{lem}

\begin{lem}\label{prop:curl-2} \label{lem-eigenvalues}  Let ${\bw}=(w_1,w_2,w_3)^T\in R^3$, $|\bw|>0$ and 
		$\bv_{1,\bw} = w/|\bw|$.
	\begin{enumerate}
	\item Assume $|{\br}_{\bw}|>0$.  Let 
\[
	{\bv}_{d,\bw} = \frac{1}{\nu_{\bw}}(\frac{s_{\bw}}{|\bw|^2} w-(1,1,1)^T - i\frac{\mu_{d,\bw}}{|\bw|^2} {\br}_{\bw} )   
\]
	\item Assume $|{\br}_{\bw}|=0$. In this case, $w_1=w_2=w_3:= \omega$.  Let 	
	\[
	v_1 = \frac{1}{\sqrt 3}(1,1,1)^T,\quad {v}_{2} := \frac{1}{\sqrt{3}}(\frac{1+i\sqrt 3}2, \frac{1-i\sqrt 3}{2}, -1)^T, \quad v_3 = \bar{V}_2.	
	\]
 Define
 \[
	v_{2,\bw} = 
	\begin{cases} 
		v_2 & \text{if } \omega \ge  0 \\ 
		v_3 & \text{if }  \omega < 0 
	\end{cases}, \quad v_{3,\bw} = \bar {v}_{2,\bw}.
\]
\end{enumerate}
Define
\[
	V_{\bw} = (v_{1,\bw},v_{2,\bw},v_{3,\bw}), \quad (\xi_{1,\omega},\xi_{2,\omega},\xi_{3,\omega}) = V_{\bw} e^{i\bx \cdot \bw}
\]
 Then   
	\begin{eqnarray}
		V_{\bw} \cdot V_{\bw}^{\dagger} &=& I_3, \label{eq:normality_general}\\	
		\CU (\xi_{d,\bw}) &=&\mu_{d,\bw} \xi_{d, \bw}, \quad d=1,2,3 \label{eq:eigen_vector_general}
	\end{eqnarray}
where $ V_{\bw}^{\dagger} =  \bar V_{\bw}^{T} $ and $\{ \mu_{d,\bw}\}_1^3$ is defined by Eq (\ref{eq:eigenvalue}.
	So $\{{\xi}_{d,\bw}\}^3_{d=1}$ are three eigenvectors of $\CU$ operator associated to $\mu_{d,\bw}$ and $\{{\bv}_{d,\bw}\}^3_{d=1}$ forms a orthogonal base of $C^3$. 
\end{lem}
\begin{rem}
	\begin{enumerate}
		\item Note that we have
		\[
		v_{1,-\bw} = -v_{1,\bw},  \quad v_{2,-\bw} = v_{3,\bw}, \quad v_{3,-\bw} = v_{2,\bw} 
		\]
		By Eq (\ref{eq:Phi_w}) and (\ref{eq:eigen_vector_general}), we have
		\begin{equation}\label{eq:Phi_w_case}
			\Phi_{\bw} v_{d,\bw} = -i\mu_{d,\bw} v_{d,\bw} = \begin{cases} 
				0 & \text{if } d=1 \\ 
				i|\bw| v_{2,\bw}  & \text{if } d=2 \\
				-i |\bw|  v_{3,\bw} & \text{if } d= 3
			\end{cases}.
		\end{equation}
		\item 
		Let
		\[
		R_{\bw} = \frac12 (v_{2,\bw} + \bar{v}_{2,\bw}),  \quad I_{\bw} = \frac1{2i} (v_{2,\bw} - \bar{v}_{2,\bw}),
		\]
		So we have
		\[
		v_{2,\bw} = R_{\bw} + iI_{\bw}, \quad v_{3,\bw} = \bar{v}_{2,\bw} = R_{\bw} - iI_{\bw}.
		\]
		One can verify directly $(R_{\bw},I_{\bw},\bw)$ forms a orthogonal base of $R^3$ for any $\bw\neq 0$ 
		\[
		\bw\cdot R_{\bw} =\bw\cdot I_{\bw}=R_{\bw}\cdot I_{\bw}=0, \quad |I_{\bw}| = |R_{\bw}| = \frac1{\sqrt 2}.
		\]
	\end{enumerate}
\end{rem}
We now consider the general case where the current density $\bJ = \bS \EE$ with $\bS$ defined by Eq (\ref{def:bs}).  To simplify the notation, we shall work in Heaviside-Lorentz units where $\eps=\mu= 1$ in the rest of the paper without loss of generosity, which is equivalent to replace $\sqrt \mu \HH$  by $\HH$,  $\sqrt{\eps}  \EE$ by $\EE$, and scale variable $t$ by $\sqrt{\mu\eps}$.  Eq (\ref{maxsys}) is equivalent to the following
\begin{equation}\label{eq:D}
	\begin{pmatrix}
		\frac{\partial {\EE}}{\partial t}   \\
		\frac{\partial{  \HH}}{\partial t}
	\end{pmatrix}
	=\BB \begin{pmatrix}
		\EE  \\
		\HH
	\end{pmatrix},   
\end{equation}
\begin{equation}\label{eq:D-init}
	\EE(\bx,0) =  E_0(\bx), \quad  \HH(\bx,0) =  H_0(\bx),
\end{equation}
where
\[
\BB= \left(
\begin{array}{cc}
	-\bS & \CU\\
	-\CU & \mathbf{0} \\
\end{array}
\right).
\] 
For any two complex vector functions $U(\bx,t),W(\bx,t)$ defined over $C^3\times [0,\infty)$, define
\[
\ZZ(\bx,t)=
\begin{pmatrix}
	U(\bx,t)  \\
	W(\bx,t)
\end{pmatrix}
\]
as a $6$ dimensional complex vector function. One can decompose $\ZZ, U, W$ into the sum of real and imaginary components, i.e. 
\[
\ZZ=\ZZ_r + i\ZZ_i,   \quad \ZZ_r= \left(
\begin{array}{c}
	 U_r\\
	W_r
\end{array}
\right),  \quad \ZZ_i=\left(
\begin{array}{c}
	U_i\\
	W_i
\end{array}
\right)
\]
Follow the idea in \cite{zou_maxwell}, we transform Maxwell's equation to abstract Cauchy Problem (ACP) and consider ACP \ref{eq:D_complex}-\ref{eq:D_complex_init} in the complex functional domain. 
\begin{eqnarray}
	\frac{\partial \ZZ(\bx, t)}{ \partial t} &=& \BB \ZZ(\bx, t),   \label{eq:D_complex} \\
	\quad \ZZ(\bx,0) &=& Z_{0}(\bx).   \label{eq:D_complex_init}   
\end{eqnarray}
where $Z_0(\bx)$ is a $6$-dim complex vector function defined on  $C^6$ and its real component is $(E_0,H_0)^T$. The solution can be expressed as \cite{engel_nagel}
\[
\ZZ(\bx, t) = e^{t\BB}Z_{0}(\bx).
\]
It is clear that $\ZZ$ solves Eq. (\ref{eq:D_complex}) if and only if both $\ZZ_r$ and $\ZZ_i$ solve Eq. (\ref{eq:D}). 

We first consider a solution when the initial value $Z_0(\bx):=Z_{0,\bw}(\bx)$ is in the $6$ dimensional complex vector space $L^{\bw}:=C^6 e^{ix\bw}$ for a given $\bw\in R^3$ and call $L^{\bw}$ as the base space associated to $\bw$.
Let 
\begin{equation}\label{def:Z}
Z_{0,\bw}(\bx)=
\begin{pmatrix}
	U  \\
	W
\end{pmatrix}=e^{i\bx\cdot \bw}\begin{pmatrix}
	X  \\
	Y
\end{pmatrix}
\in L^{\bw}, \quad X,Y \in C^3.
\end{equation}
We have
\begin{equation}\label{eq:eigen_B}
	\BB Z_{0,\bw}(\bx) = e^{i\bx \cdot \bw } \begin{pmatrix}
		-\bS X  + i \Phi_{\bw} Y   \\
		-i \Phi_{\bw}X
	\end{pmatrix} = i \left(
	\begin{array}{cc}
		i\bS & \Phi_{\bw}\\
		-\Phi_{\bw} & \mathbf{0} \\
	\end{array}
	\right)\begin{pmatrix}
		U  \\
		W
	\end{pmatrix} := i \BB_{\bw} Z_{0,\bw}(\bx),
\end{equation}
where $\Phi_{\bw}$ is defined by Eq \ref{eq:phi} and
\[
\BB_{\bw} = \left(
\begin{array}{cc}
	i\bS & \Phi_{\bw}\\
	-\Phi_{\bw} & \mathbf{0} \\
\end{array}
\right),
\]
which is treated as a linear operator on $C^6$.  With restriction of $Z_0(\bw)$ to $C^6$,   Eq. (\ref{eq:D_complex}) is then reduced to 
\begin{equation} \label{eq:ACP-w}
\frac{\partial Z(\bx, t)}{ \partial t} = i \BB_{\bw} Z(\bx, t).  
\end{equation}
To study properties of the solutions of Eq.(\ref{eq:ACP-w}), we look into the eigenvalues and associated eigenvectors of $\BB_{\bw}$. Let $(X^T,Y^T)^T$ be an eigenvector associated to an eigenvalue $\lambda$ of the matrix $\BB_{\bw}$.
Express $X,Y$ by the base $\{v_{d,\bw}\}_{d=1}^3$ in $C^3$ as
\begin{equation}\label{def:XY}
	X=\sum_{d=1}^3 a_d v_{d,\bw},  \qquad Y=\sum_{d=1}^3 b_d v_{d,\bw}.
\end{equation}
Note that $\bS$ stands for a linear operator on $C^3$ under the canonical base.  Since $V_{\bw}$ is unitary by Eq. (\ref{eq:normality_general}), $\bS$ can be transformed to $M_{\bw}$ under the base $V_{\bw}$ such that
\begin{equation}\label{def:M}
\bS V_{\bw} = V_{\bw} M_{\bw},  \qquad  M_{\bw} :=  V_{\bw}^{\dagger} \bS V_{\bw} :=(m_{d,k})_{1\le d,k\le 3}.
\end{equation}
By Eq \ref{eq:Phi_w_case}, 
$\BB_{\bw}(X^T,Y^T)^T=\lambda (X^T,Y^T)^T$ is equivalent to 
\[
e^{i\bx \cdot \bw} \begin{pmatrix}
	i\sum_{d=1}^3  v_{d,\bw} \sum_{1\le k\le 3}  a_k m_{d,k}  - i\sum_{d=1}^3 \mu_d b_d v_{d,\bw}  \\
	i\sum_{d=1}^3 \mu_d a_d v_{d,\bw} 
\end{pmatrix} = e^{i\bx \cdot \bw} 
\begin{pmatrix}
	\lambda \sum_{d=1}^3 a_d v_{d,\bw}  \\
	\lambda	\sum_{d=1}^3 b_d v_{d,\bw} 
\end{pmatrix}
\]
or
\begin{eqnarray}
	 \lambda b_d  &=& i a_d \mu_{d,\bw},  \quad 1\le d \le 3  \label{eq:implied_1}\\
	\lambda a_d   &=& -i b_d \mu_{d,\bw} + i \sum_{1\le k\le 3}  a_k m_{d,k}, \quad 1\le d\le 3. \label{eq:implied_2}
\end{eqnarray}
We now construct solutions from eigenvectors of $B_{\bw}$, which will be used as base solutions that can be used to construct solution with general initial conditions. 
\begin{lem}\label{prop:base_sol} Let $\{a_k, b_k\}_{k=1}^3, \lambda_{\bw}$ solve Eq (\ref{eq:implied_1}-\ref{eq:implied_2}), and $Z_{0,\bw}(\bx)$ be defined by Eq (\ref{def:Z}) and Eq. (\ref{def:XY}), then $Z_{0,\bw}(\bx)e^{i\lambda_{\bw} t}$ solves Eq. (\ref{eq:D_complex}) with initial condition $Z_{0,\bw}(\bx)$.
\end{lem}
\begin{proof}
	Since $(X^T,Y^T)^T$ is an eigenvector of $B_{\bw}$,  we obtain by Eq (\ref{eq:eigen_B})
	\[
	\bB (Z_{0,\bw}e^{i\lambda_{\bw} t})  = iB_{\bw} Z_{0,\bw}e^{i\lambda_{\bw} t}  = i\lambda_{\bw} Z_{0,\bw}e^{i\lambda_{\bw} t} =\frac{\partial{(Z_{0,\bw}e^{i\lambda_{\bw} t})}}{\partial t} 
	\]
\end{proof}
Note that $\mu_{1,\bw}=0$,  one can see that $\lambda=0$ is always an eigenvalue and 
\[
h_{1,\bw}:=(a_1,a_2,a_3,b_1,b_2,b_3) =(0,0,0,1,0,0)^T,
\]
is an associated eigenvector.  
As a corollary of Lemma \ref{prop:base_sol}, we have
\begin{cor}\label{prop:solution_Lw} Assume that $\{\lambda^k_{\bw}, 1\le k\le 6\}$ and
\[  
h^k_{\bw} := (a^k_{1,\bw}, a^k_{2,\bw}, a^k_{3,\bw}, b^k_{1,\bw}, b^k_{2,\bw}, b^k_{3,\bw})^T , \quad 1\le k\le 6.
\]
solve Eq (\ref{eq:implied_1}-\ref{eq:implied_2}) such that $\{h^k_{\bw}\}_{k=1}^6$ are $6$ linear independent.  Define
\begin{equation}\label{def:Z^k_w}
	X^k_{\bw}=\sum_{d=1}^3 a^k_{d,\bw} v_{d,\bw},  \quad Y^k_{\bw}=\sum_{d=1}^3 b^k_{d,\bw} v_{d,\bw}, \quad Z^k_{\bw}=e^{i\bx\cdot \bw}\begin{pmatrix}
		X^k_{\bw}  \\
		Y^k_{\bw}
	\end{pmatrix}
\end{equation}
Then  $\{Z^k_{\bw} e^{i \lambda_{k,\bw} t}\}_{k=1}^6$ solve Eq. (\ref{eq:D_complex}).  Therefore, for any initial condition $Z_{0,\bw}(\bx) \in L^{\bw}$,  the solution of Eq. (\ref{eq:D_complex}-\ref{eq:D_complex_init}) can be represented by 
\begin{equation}\label{eq:Z_w}
Z_{\bw}(\bx, t)  = \sum_{1\le k\le 6} c^k_{\bw} Z^k_{\bw} e^{i \lambda_{k,\bw} t}
\end{equation}
where the coefficients $\{c^k_{\bw}\}_{k=1}^6$ are uniquely determined by 
\[
Z_{\bw}(\bx, 0) = Z_{0,\bw}(\bx)
\]
\end{cor}
We are now ready to formulate analytic solutions of Eq. (\ref{eq:D}-\ref{eq:D-init}). 
\begin{theo}\label{thm-main}
Let $Z_0(\bx)=(\EE^T_0(\bx), \HH^T_0(\bw))^T$ be a $6$-dim complex periodic function with its Fourier series,
\begin{equation}\label{eq:complex_init}
	Z_0:=Z^R_0  + i Z^I_0= \sum_{\bw} \begin{pmatrix}
		X_{\bw}  \\
		Y_{\bw}
	\end{pmatrix} e^{i\bx \cdot \bw} = \sum_{\bw} Z_{0,\bw}(\bx)
\end{equation}
and
\begin{equation}\label{eq:real_init}
	Z^R_0 = \begin{pmatrix}
		E_0 (\bx)  \\
		H_0 (\bx)
	\end{pmatrix}.
\end{equation}
such that
	\begin{equation}\label{abs_sum_2}
		\sum_{\bw} |\bw| \sum_{1\le d\le 3} (|a_{d,\bw}| + |b_{d,\bw}| ) < \infty.
	\end{equation}
	where $a_{d,\bw},b_{d,\bw}$ are defined by Eqs (\ref{def:Z}) and (\ref{def:XY}).  Assume that for each $\bw$ in the summation \ref{eq:complex_init},  the condition in Corollary \ref{prop:solution_Lw} holds so that $\{Z^k_{\bw} \}_{k=1}^6$ by Eq. (\ref{def:Z^k_w}) is well defined. Then the solution of (\ref{eq:D_complex}-\ref{eq:D_complex_init}) is given by
	\begin{equation} \label{label:sol_allcase}
		Z(\bx, t) =  \sum_{\bw} Z_{\bw} (\bx, t) := Z^R(\bx, t) + iZ^I(\bx, t),
	\end{equation}
	where $Z_{\bw} (\bx, t)$ is  defined by Eq. (\ref{eq:Z_w}). Hence, the real component $Z^R(\bx, t)$ solves Eq. (\ref{eq:D}-\ref{eq:D-init}).
\end{theo}
\begin{proof}
	By Corollary \ref{prop:solution_Lw}, each item $Z_{\bw}(\bx, t)$ in the summation of Eq. (\ref{label:sol_allcase}) solve the equation (\ref{eq:D_complex}).  Hence $Z(\bx, t)$ solves the equation (\ref{eq:D_complex}) by the assumption (\ref{abs_sum_2}) that allows passing derivative operation into summation of the series.
\end{proof}
Since eigenvectors from different eigenvalues are always independent,  we have
\begin{cor}
Follow the notations in Theorem \ref{thm-main},  assume that for each $\bw$ in the summation \ref{eq:complex_init}, there are $6$ different eigenvalues of $\bB_{\bw}$,  then the solution of (\ref{eq:D_complex}) with the initial value $Z_0$ is $Z(\bx, t)$ defined by Eq (\ref{label:sol_allcase}) and its real component solves Eq. (\ref{eq:D}-\ref{eq:D-init}).
\end{cor}
\section{The Maxwell equation with a scalar conductivity $J=\sigma \EE$}\label{sec:j=sigmmaE}
In this section,  we assume that $\bJ$ is propositional to $E$ by applying Ohm's law \cite{bossavit} with $\bS =\sigma I_3$.  There are $4$ different eigenvalues $\BB_{\bw}$ and $6$ linear independent eigenvectors can be constructed when $\sigma\neq 2|\bw|$ as shown in Lemma \ref{lem:eigen-sigma}.  Special treatment is conducted for $\sigma\neq 2|\bw|$ in Subsection \ref{subsection-special} where the desired solution can be derived by taking limitation of the solutions as $\sigma \to 2|\bw|$.  We conclude this section by summarizing the desired analytic solution with a general initial condition on Subsection \ref{subsection-all_together}.
\subsection{The derivation of solutions in $L^{\bw}$ under $\sigma\neq 2|\bw|$}\label{subsection-general}
One can solve Eq (\ref{eq:implied_1}-\ref{eq:implied_2}) for $\lambda$ and non-zero vector $\{h^k_{\bw}\}_{k=1}^6$ as follows. The details of verification can be found in Appendix \ref{app-lem:eigen-sigma}.
\begin{lem}\label{lem:eigen-sigma}
	The eigenvalues of $\BB_{\bw}$ are given by  
	\begin{eqnarray}
		\lambda_{1,\bw} &=&0, \quad \lambda_{2,\bw} = i\sigma, \label{def:lambda_12}\\
		\lambda_{3,\bw}&=&\lambda_{4,\bw}=\frac{1}{2}(i\sigma +\sqrt{4|\bw|^2-\sigma^2}):=\lambda^+_{\bw}, \label{def:lambda_34}\\
		\lambda_{5,\bw}&=&\lambda_{6,\bw}=\frac{1}{2}(i\sigma -\sqrt{4|\bw|^2-\sigma^2}):=\lambda^-_{\bw}, \label{def:lambda_56}
	\end{eqnarray}
	with the  associated eigenvectors 
	\begin{equation}\label{def:zkw}
	Z^k_{\bw}= e^{i\bx\cdot \bw}\begin{pmatrix}
		X^k_{\bw}  \\
		Y^k_{\bw} 
	\end{pmatrix}, \quad 1\le k\le 6.  
	\end{equation}
	where $\{X^k_{\bw},Y^k_{\bw}\}_{k=1}^6$ are defined as follows.
	\begin{eqnarray*}
		\begin{pmatrix}
			X^1_{\bw}  \\
			Y^1_{\bw} 
		\end{pmatrix} &=& \begin{pmatrix}
			0  \\
			v_{1,\bw} 
		\end{pmatrix}, \quad \quad \quad 	
		\begin{pmatrix}
			X^2_{\bw}  \\ 
			Y^2_{\bw} 
		\end{pmatrix} =\begin{pmatrix}
			v_{1,\bw}  \\
			0
		\end{pmatrix}, \nonumber\\	
		\begin{pmatrix}
			X^3_{\bw}  \\
			Y^3_{\bw} 
		\end{pmatrix} &=& \begin{pmatrix}
			\lambda_{3,\bw} v_{2,\bw}  \\
			-i|\bw|v_{2,\bw}
		\end{pmatrix},\quad
		\begin{pmatrix}
			X^4_{\bw}  \\
			Y^4_{\bw} 
		\end{pmatrix} =\begin{pmatrix}
			\lambda_{4,\bw} v_{3,\bw}  \\
			i|\bw|v_{3,\bw}
		\end{pmatrix}, \label{xy_123456} \\
		\begin{pmatrix}
			X^5_{\bw}  \\
			Y^5_{\bw} 
		\end{pmatrix}&=&
		\begin{pmatrix}
			\lambda_{5,\bw} v_{2,\bw}  \\
			-i|\bw|v_{2,\bw}
		\end{pmatrix}, \quad
		\begin{pmatrix}
			X^6_{\bw}  \\
			Y^6_{\bw} 
		\end{pmatrix} = \begin{pmatrix}
			\lambda_{6,\bw} v_{3,\bw}  \\
			i|\bw|v_{3,\bw}
		\end{pmatrix}. \nonumber
	\end{eqnarray*}
\end{lem}
Under the assumption $\sigma\neq 2|\bw|$,  $\{Z^k_{\bw}\}_{k=1}^6$ form a base of $L^{\bw}$ and the analytic solution with initial $Z_{0,\bw} \in L^{\bw}$ can be derived based by Corollary \ref{prop:solution_Lw}.  
\begin{prop}\label{prop:summary-general} Let $\bw \in R^3$ such that $\sigma\neq 2|\bw|$ and 
	\begin{equation} \label{eq:init_general_sigma}
		Z_{0,\bw}(\bx) = \begin{pmatrix}
			X_{\bw}  \\
			Y_{\bw}
		\end{pmatrix} e^{i\bx \cdot \bw} =: Z^R_{0,\bw}(\bx)  + i Z^I_{0,\bw}(\bx),
	\end{equation}
be the initial condition with
	\begin{equation}\label{eq:Xw_Yw}
		X_{\bw}=\sum_{d=1}^3 a_{d,\bw} v_{d,\bw},  \quad Y_{\bw}=\sum_{d=1}^3 b_{d,\bw} v_{d,\bw}.
	\end{equation}
Then $Z_{0,\bw}(\bx)$ can be expressed as the linear combination of the base $\{Z^k_{\bw}\}^6_{k=1}$ in $L^{\bw}$
	\begin{equation}\label{xy}
		Z_{0,\bw} =\sum_{1\le k\le 6}c_{k,\bw} Z^k_{\bw} 
	\end{equation} 
where $Z^k_{\bw}$ are defined by Eq (\ref{def:zkw}) and 
	\begin{eqnarray}
		c_{1,\bw} &=& b_{1,\bw}, \quad c_{2,\bw} = a_{1,\bw}, \nonumber\\
		c_{3,\bw} &=& \frac{|\bw|a_{2,\bw} - i\lambda^-b_{2,\bw}}{|\bw|\sqrt{4|\bw|^2-\sigma^2}}, \quad c_{5,\bw} = \frac{-|\bw|a_{2,\bw} + i\lambda^+b_{2,\bw}}{|\bw|\sqrt{4|\bw|^2-\sigma^2}}\label{ck}\\
		c_{4,\bw} &=& \frac{|\bw|a_{3,\bw} + i\lambda^- b_{3,\bw}}{|\bw|\sqrt{4|\bw|^2-\sigma^2}}, \quad c_{6,\bw} = -\frac{|\bw|a_{3,\bw} + i\lambda^+b_{3,\bw}}{|\bw|\sqrt{4|\bw|^2-\sigma^2}}\nonumber
	\end{eqnarray}
	and the solution of Eq (\ref{eq:D_complex}) with the given initial condition $Z_{0,\bw}(\bx)$ is   
	\begin{equation}\label{Z_w_general}
		Z_{\bw}(\bx, t) = \sum_{1\le k\le 6} c_{k,\bw}  Z^k_{\bw} e^{i \lambda_{k,\bw} t}.
	\end{equation}
\end{prop}
where $\lambda_{k,\bw} $ are defined by Eq. (\ref{def:lambda_12}-\ref{def:lambda_56}). 
\begin{proof}
By Corollary \ref{prop:solution_Lw} ,  $Z_{\bw}(\bx, t)$ solve Eq. (\ref{eq:D_complex}).  One can directly verify $Z_{\bw}(\bx, 0) = Z_{0,\bw}(\bx) $.
\end{proof}
\subsection{The derivation of solutions in $L^{\bw}$ under  $\sigma=2|\bw|$} \label{subsection-special}
A special treatment is needed in the case $2|\bw|=\sigma$. The representation (\ref{xy}) of $Z_{0,\bw}(\bx, t)$ is not available since $\{Z^k_{\bw}\}^6_{k=1}$ no longer form a base for $L^{\bw}$.  We shall  work on $Z_{\bw}(\bx, t)$ for $\sigma<2|\bw|$ and show $\lim_{\sigma\to 2|\bw|^-}Z_{\bw}(\bx, t)$ converges to the desired solution.  In this section,  we adopt the notation $Z_{\bw}(\bx, t; \sigma)$ to specify the solution with a given $\sigma< 2|\bw|$ and we aim to show 

\begin{equation}
Z_{\bw}(\bx, t; 2|\bw|) :=  \lim_{\sigma\to 2|\bw|^-}Z_{\bw}(\bx, t; \sigma) 
\end{equation}
exists and solves Eq (\ref{eq:D}) with the initial condition defined by(\ref{eq:init_general_sigma}-\ref{eq:Xw_Yw}).
Define
\[
\eps_{\bw} = \sqrt{|\bw|^2-\frac{\sigma^2}4}, \quad \sigma_{n,\bw}=\frac{\sigma}{2|\omega|}, \quad \lambda^+_{n,\bw} = \frac{i\lambda^+_{\bw}}{|\bw|}, \quad \lambda^-_{n,\bw} = \frac{i\lambda^-_{\bw}}{|\bw|}.
\]
It is clear
\[
\lim_{\sigma\to 2|\bw|} \epsilon_{\bw} = 0, \lim_{\sigma\to 2|\bw|} \sigma_{n,\bw}= 1, \quad  \lim_{\sigma\to 2|\bw|} \lambda^+_{n,\bw} =\lim_{\sigma\to 2|\bw|} \lambda^-_{n,\bw}= -1.
\]
By definition (\ref{ck}),
\begin{eqnarray*}
	c_{3,\bw} + c_{5,\bw} &=& i\frac{b_{2,\bw}}{|\bw|}, \quad  c_{3,\bw} - c_{5,\bw} = \frac{a_{2,\bw} + b_{2,\bw}\sigma_n}{\eps}. \label{c_3_p_5}\\
	c_{4,\bw} + c_{6,\bw} &=& -i\frac{b_{3,\bw}}{|\bw|}, \quad  c_{4,\bw} - c_{6,\bw} = \frac{a_{3,\bw} - b_{3,\bw}\sigma_n}{\eps}. \label{c_4_p_6}\\
	Z^3_{\bw}-Z^5_{\bw} &=& 2\eps_{\bw}e^{i\bx\cdot\bw}\begin{pmatrix}
		v_{2,\bw} \\
		0
	\end{pmatrix}, \quad Z^4_{\bw}-Z^6_{\bw} = 2\eps_{\bw}e^{i\bx\cdot\bw}\begin{pmatrix}
		v_{3,\bw} \\
		0
	\end{pmatrix}.
\end{eqnarray*}
We shall work on the six components in Eq (\ref{Z_w_general}) for $Z_{\bw}(\bx, t; \sigma)$ separately. Let
\[
Z_{\bw}(\bx, t; \sigma) = P_{1,\bw; \sigma } + P_{2,\bw; \sigma} + P_{3,\bw; \sigma},
\]
where
\begin{eqnarray*}
P_{1,\bw; \sigma } &=& c_{1,\bw}  Z^1_{\bw} e^{i \lambda_{1,\bw} t} + c_{2,\bw}  Z^2_{\bw} e^{i \lambda_{2,\bw} t}\\
	P_{2,\bw; \sigma}&=&c_{3,\bw}  Z^3_{\bw} e^{i \lambda_{3,\bw} t} + c_{5,\bw}  Z^5_{\bw} e^{i \lambda_{5,\bw} t}, \nonumber\\ 
	P_{3,\bw; \sigma}&=&c_{4,\bw}  Z^4_{\bw} e^{i \lambda_{4,\bw} t} + c_{6,\bw}  Z^6_{\bw} e^{i \lambda_{6,\bw} t}. \nonumber
\end{eqnarray*}
 First of all, $P_{1,\bw; \sigma }$ is well-defined if $\sigma=2|\bw|$ and we have
\begin{equation}\label{prop:sol_12}
	P_{1,\bw; \sigma}= 
	\begin{pmatrix}
		a_{1,\bw}w e^{i\bx \cdot \bw -\sigma t}\\
		b_{1,\bw}w e^{i\bx \cdot \bw}
	\end{pmatrix}=:\begin{pmatrix}
		E_{1,\bw}\\
		H_{1,\bw}
	\end{pmatrix}
\end{equation} 
It turns out that $P_{2,\bw; \sigma}, P_{3,\bw; \sigma}$ converge as $\sigma\to 2|\bw|$ as follows and the proof is refereed to Appendix \ref{app-lem-limit-35-46}.
\begin{lem}\label{lem-limit-35-46}
	We have
	\begin{eqnarray}
		\lim_{\sigma\to 2|\bw| }P_{2,\bw; \sigma}&=&
		e^{- |\bw| t  +i\bx\cdot \bw} 
		\begin{pmatrix}
			(a_{2,\bw} - |\bw|(a_{2,\bw} + b_{2,\bw}) t)  v_{2,\bw}\\
			(b_{2,\bw} + |\bw|(a_{2,\bw} + b_{2,\bw}) t) v_{2,\bw}
		\end{pmatrix} := \begin{pmatrix}
			E_{2,\bw}\\
			H_{2,\bw}
		\end{pmatrix} \label{sol_35} \\
		\lim_{\sigma\to 2|\bw| }P_{3,\bw; \sigma} &=&e^{-|\bw|t+i\bx \cdot \bw}\begin{pmatrix}
			(a_{3,\bw}- |\bw|(a_{3,\bw}- b_{3,\bw})t)v_{3,\bw}\\
			(b_{3,\bw}- |\bw|(a_{3,\bw}- b_{3,\bw})t)v_{3,\bw}
		\end{pmatrix} :=   \begin{pmatrix}
			E_{3,\bw}\\
			H_{3,\bw}
		\end{pmatrix} \label{sol_46}
	\end{eqnarray} 
\end{lem}
One can directly verify that  $\{E_{2,\bw}, H_{2,\bw}\}$ and $\{E_{3,\bw}, H_{3,\bw}\}$ both solve (\ref{eq:D}) as stated below.
\begin{prop}\label{prop:special_sigma}
	Let $\{(E_{d,\bw}, H_{d,\bw})\}_{d=1}^3$ be defined by Eqs. (\ref{prop:sol_12}-\ref{sol_46}).  Then
	\begin{equation}\label{prop:sol_psi}
		\frac{\partial E_{d,\bw}}{\partial t} = -\sigma E_{d,\bw} + \CU (H_{d,\bw}), \quad \frac{\partial H_{d,\bw}}{\partial t} = -\CU (E_{d,\bw}), \quad d \in \{1,2,3\}
	\end{equation}
	and therefore 
	\begin{equation}\label{eq:Z_bw_special}
	Z_{\bw}(\bx,t; 2|\bw|) = \sum_{1\le d\le 3}  \begin{pmatrix}
		E_{d,\bw}\\
		H_{d,\bw}
	\end{pmatrix}
	\end{equation}
	is the solution of (\ref{eq:D_complex}) with the initial condition (\ref{eq:init_general_sigma}-\ref{eq:Xw_Yw}). 
\end{prop}
The details of proof is refereed to Appendix \ref{app-prop:special_sigma}.  
\begin{ex}\label{ex:parallel_field}
Note that the structure of $Z_{\bw}(\bx,t;2|\bw|)$ is different from other base solutions with $\sigma\neq 2|\bw|$ and it contains a linear component with respect to $t$. In this example, we formulate the real components of $E_{\psi,\bw},H_{\psi,\bw}$ and demonstrate how to construct electromagnetic fields with certain properties as an application.  Let
\[
E_{d,\bw} = E^R_{d,\bw} + E^I_{d,\bw}i,\quad H_{d,\bw} = H^R_{d,\bw} + H^I_{d,\bw}i
\]
By definition \ref{prop:sol_12}-\ref{sol_46} and some algebraic operations, one obtains
\begin{eqnarray}
	E^R_{1,\bw} &=& (a^R_{1,\bw} \cos(\bx\cdot \bw) - a^I_{1,\bw} \sin(\bx\cdot \bw))|\bw| e^{-2|\bw|t}   \label{E_real_12}\\
	H^R_{1,\bw} &=& (b^R_{1,\bw} \cos(\bx\cdot \bw)  -  b^I_{1,\bw} \sin(\bx\cdot \bw))  |\bw|. \label{H_real_12} \\
	E^R_{d,\bw}  &=& e^{-|\bw|t} \{(A^R_{d,\bw} \cos(\bx\cdot \bw)  - A^I_{d,\bw} \sin(\bx\cdot \bw))R_{\bw} \quad \phi \in \{1,34\} \nonumber\\ 
	&+&  (-A^R_{d,\bw} \sin(\bx\cdot \bw)  - A^I_{d,\bw} \cos(\bx\cdot \bw))I_{\bw} \} \label{E_real_psi}  \\
	H^R_{d,\bw} &=&  e^{-|\bw|t} \{ (B^R_{d,\bw} \cos(\bx\cdot \bw)  - B^I_{d,\bw} \sin(\bx\cdot \bw))R_{\bw} \quad \phi \in \{1,34\}\nonumber\\
	&+&  (-B^R_{d,\bw} \sin(\bx\cdot \bw)  - B^I_{d,\bw} \cos(\bx\cdot \bw))I_{\bw}\} \label{H_real_psi}
\end{eqnarray}
where
\begin{eqnarray*}
	A^R_{2} = a^R_{2,\bw} - t|\bw|(a^R_{2,\bw} + b^R_{2,\bw}), \quad A^I_{2} = a^I_{2,\bw} - t|\bw|(a^I_{2,\bw} + b^I_{2,\bw})\\
	B^R_{2} = b^R_{2,\bw} + t|\bw|(a^R_{2,\bw} + b^R_{2,\bw}), \quad B^I_{2} = b^I_{2,\bw} + t|\bw|(a^I_{2,\bw} + b^I_{2,\bw})\\
	A^R_{3} = a^R_{3,\bw} - t|\bw|(a^R_{3,\bw} - b^R_{3,\bw}), \quad A^I_{3} = a^I_{3,\bw} - t|\bw|(a^I_{3,\bw} - b^I_{3,\bw})\\
	B^R_{3} = b^R_{3,\bw} - t|\bw|(a^R_{3,\bw} - b^R_{3,\bw}), \quad B^I_{3} = b^I_{3,\bw} - t|\bw|(a^I_{3,\bw} - b^I_{3,\bw}).
\end{eqnarray*}
As an application,  we use the analytic expression (\ref{E_real_12}-\ref{H_real_psi}) to construct parallel electromagnetic fields.  
\begin{prop}\label{prop:application-parallel}
	For $d\in \{1,2,\}$,  assume that  $(E^R_{d,\bw},  H^R_{d,\bw})$ are not zero fields. Then they are parallel to each other if and only if
	\begin{equation}\label{iff_condition}
		a^I_{d} b^R_{d}= a^R_{d} b^I_{d}.
	\end{equation}
\end{prop}
The details of proof is referred to Appendix \ref{app-prop-application-parallel}.
\end{ex}

\subsection{The analytic solutions  with general initial conditions} \label{subsection-all_together}
The solutions on $L^{\bw}$ developed in Subsections \ref{subsection-general} and \ref{subsection-special} can be used to construct the solution with a general initial condition $Z_0(\bx)$ by the same arguments as shown in Theorem \ref{thm-main}.

\begin{theo}\label{thm-2} Assume that $\bS=\sigma I_3$ for some positive $\sigma$.
Let $Z_0(\bx)$ be a initial condition with its Fourier series defined in \ref{eq:complex_init} such that Inequity (\ref{abs_sum_2}) holds. For each $\bw$ in the summation \ref{eq:complex_init},  set $Z_{\bw}(\bx,t)$ by Eq. (\ref{Z_w_general}) if $|\bw|\neq \sigma/2$; otherwise let $Z_{\bw}(\bx,t)$ be $Z_{\bw}(\bx,t; 2|\bw|)$ by Eq. (\ref{eq:Z_bw_special}).
Then $Z(\bx, t)$ defined in Eq. (\ref{label:sol_allcase}) solves Eq. \ref{eq:D_complex}-\ref{eq:D_complex_init} and the real component $Z^R(\bx, t)$ of $Z(\bx, t)$ solves Eq. (\ref{eq:D}-\ref{eq:D-init}) .
\end{theo}

\section{The Maxwell equation with a skew-symmetric conductivity matrix}\label{sec:skew}
In this section,  we assume the conductivity matrix $\bS$ defined by Eq. (\ref{def:bs-skew}).
We shall focus on the properties of eigenvalues.
\begin{lem} For a given $\bw=(w_1,w_2,w_3)\in R^3$, let $v_{2,\bw} = (p,q,r)$, define
	\begin{equation}\label{def:phi_theta}
		\phi_{\bw} = -\frac{i}{|\bw|}(aw_1+bw_2+cw_3)=-i m \cos\alpha, \qquad \theta_{\bw} = \frac{1}{|\bw|} \left|
		\begin{array}{ccc}
			a & b & c\\
			w_1 & w_2  & w_3 \\
			p & q & r \\
		\end{array}
		\right|.
	\end{equation}
	where $\alpha$ is the angle between $\bw$ and $(a,b,c)$ such that $(aw_1+bw_2+cw_3) = m|\bw|\cos\alpha $ with
	\[
	 m^2:= a^2+b^2+c^2.
	\]
Then the transformed matrix $M_{\bw}$, defined by Eq (\ref{def:M}), is equal to
	\begin{equation}\label{eq:M_w}
		M_{\bw} = \left(
		\begin{array}{ccc}
			\sigma & \theta & \bar\theta\\
			-\bar\theta & \sigma + \phi & 0 \\
			-\theta & 0 & \sigma - \phi \\
		\end{array}
		\right).
	\end{equation}
	In addition,
	\begin{equation}\label{eq:theta_phi}
		2|\theta_{\bw}|^2 = m^2 +  \phi_{\bw}^2 = m^2- |\phi_{\bw}|^2 =m^2 \sin^2\alpha.  
	\end{equation}
\end{lem}
\begin{proof}
	Since $V_{\bw}$ is unitary, we can verify Eq (\ref{eq:M_w}) by assuming $\sigma=0$ without loss of generosity. Directly applying matrix multiplication leads to the expression (\ref{eq:M_w}) with the $\theta$ defined by Eq (\ref{def:phi_theta}).  We provide the derivation of $\phi_{\bw}$ as follows
	\begin{eqnarray*}
		\phi_{\bw} &=& a(r\bar q - \bar r q) - b(r\bar p - \bar r p )  +c (q\bar p - \bar q p)= \left|
		\begin{array}{ccc}
			a & b & c\\
			\bar p & \bar q & \bar r \\
			p & q & r \\
		\end{array}
		\right|\\
		&=& 2i \left|
		\begin{array}{ccc}
			a & b & c\\
			\mathfrak{Re}(p) & \mathfrak{Re} (q) & \mathfrak{Re} (r) \\
			\mathfrak{Im}(p) & \mathfrak{Im} (q) & \mathfrak{Im} (r) \\
		\end{array}
		\right|= \frac{2i}{\nu^2_{\bw} |\bw|^3} \left|
		\begin{array}{ccc}
			a & b & c\\
			s_{\bw}w_1-|\bw|^2 & s_{\bw}w_2-|\bw|^2  & s_{\bw}w_3-|\bw|^2 \\
			w_2-w_3 & w_3-w_1 & w_1-w_2 \\
		\end{array}
		\right|\\
		&=&\frac{2i s_{\bw}}{\nu^2_{\bw} |\bw|^3} \left|
		\begin{array}{ccc}
			a & b & c\\
			w_1 & w_2  & w_3 \\
			w_2-w_3 & w_3-w_1 & w_1-w_2 \\
		\end{array}
		\right| - \frac{2i}{\nu^2_{\bw} |\bw|} \left|
		\begin{array}{ccc}
			a & b & c\\
			1 & 1  & 1 \\
			w_2-w_3 & w_3-w_1 & w_1-w_2 \\
		\end{array}
		\right| \\
		&=& \frac{2i (s_{\bw}^2-3|\bw|^2)}{|\bw|^3\nu^2_{\bw}} (aw_1+bw_2+cw_3)= -\frac{i}{|\bw|}(aw_1+bw_2+cw_3).
	\end{eqnarray*}
	since $V_{\bw}$ is unitary,  Eq. (\ref{eq:theta_phi}) follows by observing $\bS$ and $M_{\bw}$ have same $l_2$ norm, i.e.
	\[
	2(a^2+b^2+c^2) + 3\sigma^2 = \sigma^2 + 2(\sigma + \phi) (\sigma + \bar\phi_{\bw} ) + 4\phi_{\bw}\bar\phi_{\bw} = 3\sigma^2 -2\phi^2 +4\phi\bar\phi.
	\]
\end{proof}
By Eqs (\ref{eq:implied_1}-\ref{eq:implied_2}), 
\begin{eqnarray}
	\lambda b_d &=& ia_d \mu_{d,\bw},  \quad 1\le d \le 3  \label{eq:eigen_details_1}\\
	\lambda a_1   &=& i (\sigma a_1 + \theta a_2 + \bar \theta a_3)   \label{eq:eigen_1}\\
	\lambda a_2   &=& -ib_2 \mu_{2,\bw}  + i (-\bar\theta a_1 + (\sigma + \phi) a_2 )   \label{eq:eigen_2}\\
	\lambda a_3   &=& -ib_3 \mu_{3,\bw}  + i (-\theta a_1 + (\sigma-\phi) a_3)   \label{eq:eigen_3}
\end{eqnarray}
We aim to find $6$ eigenvalues and associated eigenvectors. It is easy to verify that 
\[
\lambda = 0, \quad 	(a_1,a_2,a_3,b_1,b_2,b_3)=(0,0,0,1,0,0)
\]
solve (\ref{eq:implied_1}-\ref{eq:implied_2}). 
We denote 
\[
\lambda_1 = 0, \quad \begin{pmatrix}
	X^1_{\bw}  \\
	Y^1_{\bw} 
\end{pmatrix} = \begin{pmatrix}
	0  \\
	v_{1,\bw} 
\end{pmatrix}.
\]
We shall assume $\lambda\neq 0$ in the rest of discussion on the solution (\ref{eq:eigen_details_1}-\ref{eq:eigen_3}). By Eq (\ref{eq:eigen_details_1}), we have
\begin{equation}\label{eq:b1a2a3}
	b_1=0, \quad  a_2= i |\bw|\lambda b_2, \quad a_3= -i |\bw|\lambda b_3
\end{equation}
and Eq (\ref{eq:eigen_1}-\ref{eq:eigen_3}) is equivalent to
\begin{equation} \label{eq:Omega_a_b}
\Omega \left(
\begin{array}{c}
	a_1  \\
	b_2 \\
	b_3 \\
\end{array}
\right) = \left(
\begin{array}{c}
	0 \\
	0 \\
	0 \\
\end{array}
\right) 
\end{equation}
where
\[
\Omega = \left(
\begin{array}{ccc}
	i\sigma -\lambda &  -\theta\lambda/|\bw|  &  \bar\theta\lambda/|\bw| \\
	-|\bw| \bar\theta &  -\lambda^2+ i(\sigma+\phi)\lambda + |\bw|^2   & 0 \\
	|\bw| \theta & 0 & -\lambda^2+ i(\sigma-\phi)\lambda + |\bw|^2  \\
\end{array}
\right).
\]
One can calculate the determinant of $\Omega$
\begin{eqnarray}
|\Omega| &=& (i\sigma -\lambda )(-\lambda^2+ i(\sigma+\phi)\lambda + |\bw|^2 )( -\lambda^2+ i(\sigma-\phi)\lambda + |\bw|^2) \nonumber\\
&+& 2|\theta|^2 \lambda(\lambda^2 - i\sigma\lambda -|\bw|^2)\label{eq:Q_1}\\
&=& (i\sigma -\lambda )(\lambda^2 -i\sigma\lambda -|\bw|^2 )^2 -\phi^2 \lambda^2 (\lambda-i\sigma) + 2|\theta|^2 \lambda(\lambda^2 - i\sigma\lambda -|\bw|^2) \label{eq:Omega} \label{eq:Q_2}\\
&=& -\lambda^5 + 3i\sigma\lambda^4 + (3\sigma^2 + 2|\bw|^2 + m^2) \lambda^3 \nonumber\\
&-& i\sigma (\sigma^2 + 4|\bw|^2 + m^2)\lambda^2 - |\bw|^2 (2\sigma^2 + |\bw|^2 + 2|\theta|^2) \lambda + i\sigma|\bw|^4. \label{eq:Q_3}
\end{eqnarray}
Substitute  $\lambda = iy$ in Eq (\ref{eq:Q_2}),  $|\Omega|=0$ implies 
\begin{equation}\label{def:gy}
g(y):=(\sigma -y)(y^2 -\sigma y +|\bw|^2 )^2 +\phi^2 y^2 (y-\sigma) -2|\theta|^2 y(y^2 - \sigma y +|\bw|^2)=0 
\end{equation}
it is clear that there is at least one  real solution $y>0$. Replace $\phi^2=-m^2 \cos^2\alpha$ and $2|\theta|^2 = m^2 \sin^2\alpha$,  Eq. (\ref{def:gy}) is reduced to 
\begin{equation}\label{eq:sin}
\sin^2\alpha = \frac{(\sigma - y)(m^2y^2 + (y^2-\sigma y + |\bw|^2)^2 )}{m^2 |\bw|^2 y}
\end{equation}
\begin{prop}
There is at least one solution of $|\Omega=0|$ such that $\lambda=iy$ for a positive $y>0$. Furthermore, any such solution must satisfy
\begin{equation}
	  \frac{|\bw|^2 (\sigma -\frac m2)}{\sigma^2 + |\bw|^2} \le y \le \sigma
\end{equation}
\end{prop}
\begin{proof}
Note $g(0)>0$ and $\lim_{y\to \infty}g(y) = -\infty$,  which implies there is at least one positive solution. For any positive solution $y$,  it is clear $y\le \sigma$ by Eq (\ref{eq:sin}).  Eq (\ref{eq:sin}) implies
\[
1\ge   \frac{2(\sigma - y)my(y^2-\sigma y + |\bw|^2)}{m^2 |\bw|^2 y} = \frac{2(\sigma - y)(y^2-\sigma y + |\bw|^2) )}{m |\bw|^2 } 
\] 
and therefore
\[
y \ge \frac{|\bw|^2 (\sigma-\frac m2)}{(y-\sigma)^2 + |\bw|^2} \ge \frac{|\bw|^2 (\sigma-\frac m2)}{\sigma^2 + |\bw|^2} 
\]
\end{proof}
By Eq (\ref{eq:Q_3}),  one can show the following property and the details can be found in Section \ref{app-prop:real_solution}.
\begin{prop}\label{prop:real_solution} There exists real solution $\lambda$ for $|\Omega|=0$ if and only if $\sigma=0$. 
\end{prop}
The close form solution $\{\lambda_k\}_{k=2}^6$ is available in two special cases.
\begin{enumerate}
	\item $\alpha=\frac{\pi}2$. In this case, $\phi=0, 2|\theta|^2 = m^2>0$,  define
	\[
	\tau^2 = 2|\theta|^2 + |\bw|^2 + \sigma^2 = m^2 + |\bw|^2 + \sigma^2 
	\]
By Eq (\ref{eq:Omega}), $|\Omega|=0$ implies 
	\begin{eqnarray*}
		(i\sigma -\lambda )(\lambda^2 -i\sigma\lambda -|\bw|^2 )^2 + 2|\theta|^2 \lambda(\lambda^2 - i\sigma\lambda -|\bw|^2) =0
	\end{eqnarray*}
Let $\lambda_3$ and $\lambda_5$ denote the two solutions of  
	\[
		\lambda^2 -i\sigma\lambda -|\bw|^2 =0.
		\]
We obtain
		\begin{eqnarray*}
		\lambda_3 = \frac 1{2} (i\sigma + \sqrt{4|\bw|^2 -\sigma^2}), \quad  \lambda_5 = \frac 12 (i\sigma - \sqrt{4|\bw|^2 -\sigma^2}),
	\end{eqnarray*}
which recovers the two eigenvalues in the case $\bJ = \sigma I_3$ as shown by Eqs. (\ref{def:lambda_34}-\ref{def:lambda_56}).
The other three solutions $\lambda_2, \lambda_4$ and $\lambda_6$,  solves the following equation
\begin{equation}\label{eq:lambda246}
	\lambda^3 -2 i\sigma\lambda^2  - \tau^2\lambda + i\sigma |\bw|^2  =0
\end{equation}
Let $\lambda=i y$, we can transform Eq (\ref{eq:lambda246}) to a equation with real coefficients
\begin{equation}\label{eq:y}
	g(y):=y^3 -2 \sigma y^2  + \tau^2 y  - \sigma |\bw|^2  =0.
	\end{equation}
The solution structure of Eq. (\ref{eq:y}) is summarized below and the proof is refereed to Appendix \ref{app-lem:y_structure}.
\begin{lem}\label{lem:y_structure} 
\begin{enumerate}
	\item The unique real solution of Eq. (\ref{eq:y}) with multiplicity $3$ occurs at  $y_1=y_2=y_3=\frac{\sqrt 6}{2} |\bw|$ when $\sigma=\frac{3\sqrt6}{2}|\bw|$ and $m=\frac{\sqrt 2}{4}|\bw|$. 
	\item The unique real solution of Eq. (\ref{eq:y}) with multiplicity $2$ occurs at $\hat y := \frac{9|\bw|^2-2\tau}{2(3\tau - 4\sigma^2)} \sigma$ when $3\tau - 4\sigma^2>0$ and $g(\hat y) = g'(\hat y)=0$.
	\item In other cases, there are three different solutions $y_1,y_2,y_3$ such that $y_1$ is real and $0<y_1<\sigma$. Furthermore, if $y_2, y_3 $ are complex, then $y_3 = \bar{y_2}$.
\end{enumerate}
\end{lem}
\item $\alpha=0$.  In this case, $\theta=0, \phi = -im $.  The solutions include $i\sigma$ and other four different complex numbers if $m \neq 0$. 
\begin{eqnarray*}
& & \frac12 (i(\sigma-im) + \sqrt{4|\bw|^2 - (\sigma-im)^2}), \quad \frac12 (i(\sigma-im) - \sqrt{4|\bw|^2 - (\sigma-im)^2})\\
& & \frac12 (i(\sigma +im) + \sqrt{4|\bw|^2 - (\sigma+im)^2}), \quad \frac12 (i(\sigma +im) - \sqrt{4|\bw|^2 - (\sigma+im)^2}).
\end{eqnarray*}
Note that the solution $i\sigma$ also apply to the case $\bJ = \sigma I_3$ as shown by Eq. (\ref{def:lambda_12}).   If $m=0$,  all solutions are reduced to the solutions to the case $\bJ=\sigma I_3$ as expected.
\end{enumerate}
\section{The Maxwell equation with an independent local electromagnetic field}\label{sec:jg}
In this section, we consider a general case by adding term $J^g$ as part of the current density $\bJ$ as expressed in Eq. (\ref{def:Sigma}).  Assume that the conductivity matrix $\bS$ has the inverse $\boldsymbol {\rho } = \bS^{-1}$. We want to provide the solution for following Maxwell equations.
\begin{eqnarray}
\frac{\partial \HH}{\partial t} &=&-\CU (\EE), \qquad  \frac{\partial \EE}{\partial t} = \CU  (\HH) -  \bS \EE -  J^g(\bx), \label{maxsys a-2}\\
\HH(\bx,0) &=& H_0(\bx) ,\qquad  \EE(\bx,0) =E_0(\bx), \label{IV-2}
\end{eqnarray}
With the tools developed in previous sections,  it turns out that the desired solutions can be directly constructed as follows. 
\begin{theo}\label{theorem-main-3}
Let $J^g(\bx)$ be continuous differentiable and can be represented by its Fourier series as follows
\[
J^g(\bx) = \sum_{\bw} a_{\bw}e^{i\bx\cdot \bw} = v_0 + \sum_{\bw\neq 0} \sum_{1\le d \le 3} \alpha_{d,\bw}\bv_{d,\bw} e^{i\bx\cdot \bw}.
\]
where $v_0$ is a constant vector in $C^3$.
Define 
\[
\psi(\bx)= \sum_{\bw \neq 0} \sum_{2\le d \le 3} \frac{1}{\mu_{d,\bw}} \alpha_{d,\bw}\bv_{d,\bw} e^{i\bx\cdot \bw} 
\]
and
\[
\phi(\bx) = \brho (v_0 + \sum_{\bw \neq 0} \alpha_{1,\bw}\bv_{1,\bw} e^{i\bx\cdot \bw})
\]
Let $U, V$ be the solution 
\begin{eqnarray*}
	\frac{\partial V}{\partial t} &=&-\CU (U), \qquad  \frac{\partial U}{\partial t} = \CU  (V) -  \bS U, \label{maxsys a-home}\\
	V(\bx,0) &=& H_0(\bx) - \psi(\bx) ,\qquad  U(\bx,0) = E_0(\bx) + \phi(\bx)  
\end{eqnarray*}
Then the desired solution for (\ref{maxsys a-2}) and (\ref{IV-2}) is given by
\[
\HH = V(\bx, t) +\psi(\bx), \quad \EE= U(\bx, t)-\phi(\bx).
\]
\end{theo}
\begin{proof}
It is easy to see that
\[
\frac{\partial \psi} {\partial t}=0, \quad
\CU (\psi) =\sum_{\bw \neq 0} \sum_{2\le d \le 3}  \alpha_{d,\bw}\bv_{d,\bw} e^{i\bx\cdot \bw}  = J^g(\bx) - \bS \phi(\bx) 
\]
and
\[
\CU  (\phi) =0, \quad   \frac{\partial \phi} {\partial t} = 0
\]
 we have
	\begin{eqnarray*}
		\dot \EE &=& \dot U = \CU (V)-\bS U = \CU (\HH-\psi) - \bS U  = \CU (\HH) -\sigma \EE - J^g\\
		\dot \HH &=& \dot V = -\CU (U) = - \CU (\EE + \phi) = -\CU (\EE)
	\end{eqnarray*}

\end{proof}

\appendix
\section{The proves of Lemmas \ref{lem:eigen-sigma}, \ref{lem-limit-35-46}, \ref{lem-limit-35-46}, \ref{lem:y_structure}, and Propositions \ref{prop:special_sigma}, \ref{prop:application-parallel}}
\subsection{The proof of Lemma \ref{lem:eigen-sigma}}\label{app-lem:eigen-sigma}
\begin{proof}
	We follow the notations in Section \ref{sec:j=sigmmaE}. If $Z=(X^T,Y^T)^T e^{-ix\bw} \in L^{\bw}$ is an eigenvector associated to certain eigenvalue $i\lambda$ of $\BB$, then $(X^T,Y^T)^T$ is the eigenvalue associated to the eigenvalue $\lambda$ of the matrix $\BB_{\bw}$. 
 By Eq. \ref{eq:Phi_w_case}, 
	\begin{eqnarray*}
		\BB_{\bw} \begin{pmatrix}
			X^1_{\bw}  \\
			Y^1_{\bw}
		\end{pmatrix} &=&  \begin{pmatrix}
			\Phi_{\bw}  v_{1,\bw}\\
			0
		\end{pmatrix} = \begin{pmatrix}
			0  \\
			0
		\end{pmatrix} =  \lambda_{1,\bw}\begin{pmatrix}
			X^1_{\bw}  \\
			Y^1_{\bw}
		\end{pmatrix}, \\
		\BB_{\bw} \begin{pmatrix}
			X^2_{\bw}  \\
			Y^2_{\bw}
		\end{pmatrix} &=&  \begin{pmatrix}
			i\sigma   v_{1,\bw}\\
			-\Phi_{\bw}  v_{1,\bw}
		\end{pmatrix} = \begin{pmatrix}
			i\sigma   v_{1,\bw}  \\
			0
		\end{pmatrix} =  \lambda_{2,\bw}\begin{pmatrix}
			X^2_{\bw}  \\
			Y^2_{\bw}
		\end{pmatrix}, \\
		\BB_{\bw} \begin{pmatrix}
			X^3_{\bw}  \\
			Y^3_{\bw}
		\end{pmatrix} &=&  \begin{pmatrix}
			(i\sigma \lambda_{3,\bw} + |\bw|^2 )  v_{2,\bw} \\
			-i|\bw| \lambda_{3,\bw} v_{2,\bw} 
		\end{pmatrix} = \begin{pmatrix}
			(i\sigma \lambda_{3,\bw} - \lambda_{3,\bw}{\lambda}^-_{\bw} )  v_{2,\bw} \\
			-i|\bw| \lambda_{3,\bw} v_{2,\bw} 
		\end{pmatrix} =  \lambda_{3,\bw}\begin{pmatrix}
			X^3_{\bw}  \\
			Y^3_{\bw}
		\end{pmatrix}, \\
		\BB_{\bw} \begin{pmatrix}
			X^4_{\bw}  \\
			Y^4_{\bw}
		\end{pmatrix} &=&  \begin{pmatrix}
			(i\sigma \lambda_{4,\bw} + |\bw|^2 )  v_{3,\bw} \\
			i|\bw| \lambda_{4,\bw} v_{3,\bw} 
		\end{pmatrix} = \begin{pmatrix}
			(i\sigma \lambda_{4,\bw} - \lambda_{4,\bw}{\lambda}^-_{\bw} )  v_{3,\bw} \\
			i|\bw| \lambda_{4,\bw} v_{3,\bw} 
		\end{pmatrix} =  \lambda_{4,\bw}\begin{pmatrix}
			X^4_{\bw}  \\
			Y^4_{\bw}
		\end{pmatrix} \nonumber\\
		\BB_{\bw} \begin{pmatrix}
			X^5_{\bw}  \\
			Y^5_{\bw}
		\end{pmatrix} &=&  \begin{pmatrix}
			(i\sigma \lambda_{5,\bw} + |\bw|^2 )  v_{2,\bw} \\
			-i|\bw| \lambda_{5,\bw} v_{2,\bw} 
		\end{pmatrix} = \begin{pmatrix}
			(i\sigma \lambda_{5,\bw} - \lambda_{5,\bw}\lambda^+_{\bw} )  v_{2,\bw} \\
			-i|\bw| \lambda_{5,\bw} v_{3,\bw} 
		\end{pmatrix} =  \lambda_{5,\bw}\begin{pmatrix}
			X^5_{\bw}  \\
			Y^5_{\bw}
		\end{pmatrix} \\
		\BB_{\bw} \begin{pmatrix}
			X^6_{\bw}  \\
			Y^6_{\bw}
		\end{pmatrix} &=&  \begin{pmatrix}
			(i\sigma \lambda_{6,\bw} + |\bw|^2 )  v_{3,\bw} \\
			i|\bw| \lambda_{6,\bw} v_{3,\bw} 
		\end{pmatrix} = \begin{pmatrix}
			(i\sigma \lambda_{6,\bw} - \lambda_{6,\bw}\lambda^+_{\bw} )  v_{3,\bw} \\
			i|\bw| \lambda_{6,\bw} v_{3,\bw} 
		\end{pmatrix} =  \lambda_{6,\bw}\begin{pmatrix}
			X^6_{\bw}  \\
			Y^6_{\bw}
		\end{pmatrix} \nonumber
	\end{eqnarray*}
So $\lambda_{k,\bw}$ is a eigenvalue of $\BB_{\bw}$ and $Z^k_{\bw}$ is the associated  eigenvector for $1\le k \le 6$.
\end{proof}
\subsection{The proof of Lemma \ref{lem-limit-35-46}}\label{app-lem-limit-35-46}
\begin{proof}
We have
\begin{eqnarray*}
	& & P_{2,\bw} = e^{-\frac {\sigma t}2}\{
	Z^5_{\bw} (c_{3,\bw} + c_{5,\bw})\cos(\eps t) + i Z^5_{\bw} \sin(\eps t) (c_{3,\bw} - c_{5,\bw}) \nonumber\\
	&+& 2\eps c_{3,\bw} (\cos(\eps t) + i\sin(\eps t)) \begin{pmatrix}
		v_{2,\bw}\\
		0
	\end{pmatrix} \}\\
	&=& e^{-\frac {\sigma t}2 +i\bx\cdot \bw}\{
	\cos(\eps t)b_{2,\bw}  \begin{pmatrix}
		\lambda^-_{n,\bw}v_{2,\bw}\\
		v_{2,\bw}
	\end{pmatrix} + |\bw|(a_{2,\bw} + \sigma_n b_{2,\bw})\frac{\sin(\eps_{\bw} t)}{\eps_{\bw}} \begin{pmatrix}
		\lambda^-_{n,\bw}v_{2,\bw}\\
		v_{2,\bw}
	\end{pmatrix} \nonumber\\
	&+& (a_{2,\bw} - b_{2,\bw}\lambda^-_{n,\bw}) e^{\epsilon t i}\begin{pmatrix}
		v_{2,\bw}\\
		0
	\end{pmatrix} \}.
\end{eqnarray*}
Taking limit $\sigma\to 2|\bw|$, we obtain Eqs (\ref{sol_35}). Similarly, one can show Eq. (\ref{sol_46}).  
\end{proof}
\subsection{The proof of Proposition \ref{prop:special_sigma}}\label{app-prop:special_sigma}
\begin{proof} We have
	\begin{eqnarray*}
		E_{2,\bw} + H_{2,\bw} &=& (a_{2,\bw} + b_{2,\bw})v_{3,\bw}e^{i\bx\cdot \bw-|\bw|t},\\
		E_{3,\bw} - H_{3,\bw} &=& (a_{3,\bw} - b_{3,\bw})v_{3,\bw}e^{i\bx\cdot \bw-|\bw|t}.
	\end{eqnarray*}
	and
	\begin{eqnarray*}
		\CU (E_{2,\bw}) &=& -|\bw| E_{2,\bw},  \quad \CU (H_{2,\bw}) = -|\bw| H_{2,\bw},\\
		\CU (E_{3,\bw}) &=& |\bw| E_{2,\bw},  \quad \CU (H_{3,\bw}) = |\bw| H_{3,\bw}.
	\end{eqnarray*}
The case $d=1$ can be easily checked since $\CU (E_{1,\bw})= \CU (H_{1,\bw})=0$. For $d=2$, 
	\begin{eqnarray*}
		\frac{\partial E_{2,\bw}}{\partial t}&=& -|\bw|  E_{2,\bw} -  |\bw|(a_{2,\bw} + b_{2,\bw})v_{2,\bw}e^{i\bx\cdot \bw-|\bw|t}=-|\bw|  E_{2,\bw}  -|\bw|(E_{2,\bw} + H_{2,\bw}) \nonumber\\ 
		&=& -2|\bw| E_{2,\bw} + \CU H_{2,\bw} = -\sigma E_{2,\bw} + \CU (H_{2,\bw}) \\
		\frac{\partial H_{2,\bw}}{\partial t}&=& -|\bw|  H_{2,\bw} +  |\bw|(a_{2,\bw} + b_{2,\bw})v_{2,\bw}e^{i\bx\cdot \bw-|\bw|t}=-|\bw|  H_{2,\bw}  +|\bw|(E_{2,\bw} + H_{2,\bw}) \nonumber\\ 
		&=&  -\CU (H_{2,\bw}),
	\end{eqnarray*}
	Similarly for $d=3$,
	\begin{eqnarray*}
		\frac{\partial E_{3,\bw}}{\partial t}&=& -|\bw|  E_{3,\bw} -  |\bw|(a_{3,\bw} - b_{3,\bw})v_{2,\bw}e^{i\bx\cdot \bw-|\bw|t}=-|\bw|  E_{3,\bw}  -|\bw|(E_{3,\bw} - H_{3,\bw}) \nonumber\\
		&=& -2|\bw| E_{3,\bw} + \CU H_{3,\bw} = -\sigma E_{3,\bw} + \CU (H_{3,\bw}) \\
		\frac{\partial H_{3,\bw}}{\partial t}&=& -|\bw|  H_{3,\bw} -  |\bw|(a_{3,\bw} - b_{3,\bw})v_{3,\bw}e^{i\bx\cdot \bw-|\bw|t}=-|\bw|  H_{3,\bw}  -|\bw|(E_{2,\bw} - H_{2,\bw}) \nonumber\\ 
		&=&  -\CU (H_{3,\bw}).
	\end{eqnarray*}
Therefore,  $Z_{\bw}(\bx,t; 2|\bw|)$ solves Eq. (\ref{eq:D_complex}) and meets the initial condition since 
\[
\lim_{\sigma\to 2|\bw|}Z_{0,\bw; \sigma}=Z_{0,\bw; 2|\bw|}.
\]
\end{proof}

\subsection{The proof of Proposition  \ref{prop:application-parallel}} \label{app-prop-application-parallel}
\begin{proof}
	One can directly check the sufficient condition.  We prove the necessary condition for $d=2$ and same argument can be applied for other two cases. Without loss of generosity, we assume $E^R_{2}$ is not a zero vector field.  If $H^R_{2}$ is parallel to $E^R_{2}$, then $H^R_{2} =c(\bx,t) E^R_{2}$ for some real-value function $c (\bx, t)$, which implies by Eqs (\ref{E_real_psi}-\ref{H_real_psi})
	\begin{eqnarray*}
		B^R_{2} \cos(\bx\cdot \bw)  - B^I_{2} \sin(\bx\cdot \bw) &=& c (A^R_{2} \cos(\bx\cdot \bw)  - A^I_{2} \sin(\bx\cdot \bw))\\,
		-B^R_{2} \sin(\bx\cdot \bw)  - B^I_{2} \cos(\bx\cdot \bw) &=& c( (-A^R_{2} \sin(\bx\cdot \bw)  - A^I_{2} \cos(\bx\cdot \bw)) )
	\end{eqnarray*}
	which implies
	\[
	B^R_{2} = c A^R_{2}, \quad B^I_{2} = c A^I_{2},
	\]
	or equivalently,
	\begin{eqnarray}
		t|\bw|(a^R_{2,\bw} + b^R_{2,\bw}) (1+c) &=& c a^R_{2,\bw} -b^R_{2,\bw} \label{cases_1} \\
		t|\bw|(a^I_{2,\bw} + b^I_{2,\bw}) (1+c) &=& c a^I_{2,\bw}  - b^I_{2,\bw}  \label{cases_2} 
	\end{eqnarray}
	Eq. (\ref{iff_condition}) clearly holds if $c=-1$. We shall show Eq. (\ref{iff_condition}) under the assumption $c\neq -1$.
	\begin{enumerate}
		\item Assume $c a^I_{2,\bw} -b^I_{2,\bw} =0$.  Eq. (\ref{cases_2}) implies that $a^I_{2,\bw} + b^I_{2,\bw}=0$.  So we have $b^I_{2,\bw} = -a^I_{2,\bw}$ and hence $a^I_{2,\bw}=0$ by Eq. (\ref{cases_2}) and hence $b^I_{2,\bw}=0 $, which implies Eq. (\ref{iff_condition}).
		\item  Assume $c a^I_{2,\bw} -b^I_{2,\bw} \neq 0$.  We have by Eqs. (\ref{cases_1}-\ref{cases_2})
		\[
		(a^R_{2,\bw} + b^R_{2,\bw}) (c a^I_{2,\bw}  - b^I_{2,\bw}) = (a^I_{2,\bw} + b^I_{2,\bw})(c a^R_{2,\bw} -b^R_{2,\bw})
		\]
		or
		\[
		(c+1)(a^R_{2,\bw} b^I_{2,\bw} - a^I_{2,\bw} b^R_{2,\bw})=0
		\]
		which implies Eq. (\ref{iff_condition}).
	\end{enumerate}
\end{proof}

\subsection{The proof of Proposition  \ref{prop:real_solution}} \label{app-prop:real_solution}
\begin{proof}
If $\sigma=0$, the solutions of Eq (\ref{eq:Q_1}) include $0$ and following four different real numbers with $m>0$.
\begin{eqnarray*}
& & \frac{1}{\sqrt 2}\sqrt{2|\bw|^2+m^2+\sqrt{m^4 + 4|\bw|^2 m^2 \cos^2\alpha  }},\quad  - \frac{1}{\sqrt 2}\sqrt{2|\bw|^2+m^2+\sqrt{m^4 + |\bw|^2 m^2 \cos^2\alpha}} \\
& & \frac{1}{\sqrt 2}\sqrt{2|\bw|^2+m^2-\sqrt{m^4 + 4|\bw|^2 m^2 \cos^2\alpha}} ,\quad - \frac{1}{\sqrt 2}\sqrt{2|\bw|^2+m^2-\sqrt{m^4 + 4|\bw|^2 m^2 \cos^2\alpha}}.
\end{eqnarray*}
Now assume $\sigma\ne 0$ and $\lambda$ be a real solution of Eq (\ref{eq:Q_1}).  It is clear that $\lambda\neq 0$ and we would obtain by Eq (\ref{eq:Q_3})
\begin{eqnarray}
	0&=&\lambda^4 -  (3\sigma^2 + 2|\bw|^2 + m^2) \lambda^2 + |\bw|^2 (2\sigma^2 + |\bw|^2 + 2|\theta|^2) \label{A-3} \\
	0&=&3\lambda^4 - (\sigma^2 + 4|\bw|^2 + m^2)\lambda^2 + |\bw|^4   \label{A-4}
\end{eqnarray}
Applying $(\ref{A-4}) -(\ref{A-3})$, we have
\begin{equation}\label{A-5}
	(\lambda^2 + \sigma^2)(\frac{\lambda^2}{|\bw|^2}-1)= |\theta|^2
\end{equation}
Applying $(\ref{A-4})-3(\times\ref{A-3})$, we have
\begin{equation}\label{A-6}
	\frac{\lambda^2}{|\bw|^2}= \frac{3\sigma^2 + |\bw|^2 + 3 |\theta|^2}{4\sigma^2 + |\bw|^2 + m^2}
\end{equation}
Applying Eq (\ref{A-6})  to Eq  (\ref{A-5}), we have
\begin{equation}\label{A-7}
	\lambda^2 = \frac{|\theta|^2 s_3 + s_4}{3 |\theta|^2  - \tau}
\end{equation}
where 
\begin{eqnarray*}
\tau &=&\sigma^2+m^2, \quad	s_1 = \frac{3\sigma^2 + |\bw|^2 + \tau}{|\bw|^2}\\
 s_2 &=& 3\sigma^2+|\bw|^2, \quad s_3= \tau + |\bw|^2, \quad s_4 = \sigma^2 \tau, \quad s_5 = 3s_2 -3\tau-s_1s_3 
\end{eqnarray*}
By Eqs (\ref{A-6}-\ref{A-7}), 
\begin{equation}
	\frac{3 |\theta|^2 + s_2}{s_1} = \frac{s_3|\theta|^2 +s_4 }{3 |\theta|^2 -\tau}
\end{equation}
so
\begin{equation}\label{A-8}
|\theta|^2 = \frac{-s_5 + \sqrt{s_5^2 + 36(s_2\tau +s_1s_4)}}{18}
\end{equation}
applying $|\theta|^2= \frac12 m^2 \sin^2\alpha$, Eq (\ref{A-8}) implies
\[
s_5^2 + 9m^2 \ge \sqrt{s_5^2 + 36(s_2\tau +s_1s_4)}
\]
which is equivalent to
\begin{equation}
m^4 \ge 12 \sigma^4 + 4\sigma^2 |\bw|^2 + 8 m^2 \sigma^2 + 2s_1(m^2 + \sigma^2)(m^2 + 2\sigma^2) > m^4
\end{equation}
since $s_1>1$.  The contradiction shows that no real solution exists. 
\end{proof}

\subsection{The proof of Lemma \ref{lem:y_structure}} \label{app-lem:y_structure}
\begin{proof} Note that $g(\sigma)=\sigma m^2>0$ and $g(0)<0$.  We have 
\[
g'(y)=(3y -\sigma)(y-\sigma) + m^2+|\bw|^2> 0 ,  \quad y\ge \sigma
\]
and therefore $g(y)$ is increasing as $y>\sigma$, hence any positive real solution of $g(y)$ must be in the range $(0,\sigma ]$. 
\begin{enumerate}
	\item If $\sigma=\frac{3\sqrt6}{2}|\bw|$ and $m=\frac{\sqrt 2}{4}|\bw|$,  one can directly verify that $g(y)=(y-\hat y)^3$. Also,  the solution with multiplicity $3$ must meet $g''(\hat)=0$, implies that $\hat y = \frac 23 \sigma$ is the unique solution with multiplicity $3$. 
\item 
We now assume that there is a solution $\hat y$ with multiplicity $2$, i.e,  $g(y)= (y-y_1)(y-\hat y)^2$. Then $y_1\neq \hat y$ and $y_1,\hat y$ must be real.  Since $h'(\hat y)=h(\hat y)=0$,  which implies 
\begin{eqnarray}
0&=&\hat y^3 - 2\sigma \hat y^2  +\tau \hat y -\sigma |\bw|^2  \label{append:A_9}\\
0&=&\hat y^2 - \frac{4\sigma}{3} \hat y + \frac \tau 3  \label{append:A_10}
\end{eqnarray}
Applying $\hat y \times Eq. (\ref{append:A_10}) - Eq. (\ref{append:A_9})$, we obtain
\begin{equation}\label{append:A_11}
\hat y^2 -\frac{\tau}{\sigma} \hat  y + \frac{3}{2} |\bw|^2 = 0
\end{equation}
which implies
\[
\hat y = \frac{\frac{4\sigma}{3} \pm \sqrt{(\frac{4\sigma}{3})^2 - \frac{4\tau }3}}{2},
\]
and therefore
\[
4\sigma^2 -3 \tau \ge  0 
\]
Applying $Eq (\ref{append:A_10})-Eq(\ref{append:A_11})$, we have
\[
{2(4\sigma^2 -3\tau)} \hat y = (2\tau - 9|\bw|^2) \sigma
\]
If $4\sigma^2 -3 \tau =  0$, then $2\tau = 9|\bw|^2$,  which is reduced to case 1.  So we have
$4\sigma^2 -3 \tau > 0$ and
\[
 \hat y = \frac{2\tau - 9|\bw|^2}{2(4\sigma^2 -3\tau)} \sigma
\]
Therefore,  $\hat y$ is the solution with multiplicity $2$ if and only if $g(\hat y)=g'(\hat y)=0$ and $\hat y$ is defined as above.
\item For any other case,  there are three different solutions by definition.
\end{enumerate}
\end{proof}

\end{document}